\documentclass[11pt]{amsart}
\usepackage{amsmath}
\usepackage{amssymb} % for \mathbb

\def\beqnn{\begin{eqnarray*}}\def\eeqnn{\end{eqnarray*}}
\newtheorem{theorem}{Theorem}[section]

\newtheorem{proposition}[theorem]{Proposition}

\theoremstyle{definition}

\theoremstyle{question}

\numberwithin{equation}{section}

\begin{document}

\title{Generalized Ces\`aro operators on Dirichlet-type spaces}

%\author{Huabing Li}
%\address{School of Mathematics Sciences, Hefei University of Technology, Xuancheng Campus, Xuancheng 242000, P.R.China}
%\email{musicli121@163.com}

\author{Jianjun Jin}
\address{School of Mathematics Sciences, Hefei University of Technology, Xuancheng Campus, Xuancheng 242000, P.R.China}
\email{jinjjhb@163.com, jin@hfut.edu.cn}

% \author{Xiaogao Feng}
%   \address{College of Mathematics and Information, China West Normal University, Nanchong 637009, P.R.China}
%  %    Current address
 % \email{fengxiaogao603@163.com}

  \author{Shuan Tang}
  \address{School of Mathematics Sciences, Guizhou Normal University, Guiyang 550001,
  P.R.China} \email{tsa@gznu.edu.cn}

   % \author{Chunhua Li}
  %\address{School of Computer Science and Information Engineering, Hefei University of
%Technology, Hefei 230009, China} \email{lch2014@hfut.edu.cn}

\thanks{The first author supported by National Natural Science
Foundation of China (Grant Nos. 11501157).
% The second author was supported by National Natural Science Foundation of China (Grant
%Nos. 11701459). }
% The third author
 The second author was supported by National Natural Science Foundation of China (Grant
Nos. 12061022) and the foundation of Guizhou Provincial Science and
Technology Department (Grant Nos. [2017]7337 and [2017]5726).}
%The third author was supported by National Natural Science Foundation of China (Grant
%Nos. 11705043,
%11547137)
%\author{Author Two}
%\address{Mathematical Research Section, School of Mathematical Sciences,
%Australian National University, Canberra ACT 2601, Australia}
%\thanks{This paper has been accepted by  Acta Mathematica Scientia.}

%    General info

%\date{February 9, 2013.}



\begin{abstract}
In this note, we introduce and study a new kind of generalized Ces\`aro operators $\mathcal{C}_{\mu}$, induced by a positive Borel measure $\mu$ on $[0, 1)$, between the Dirichlet-type spaces. We characterize the measures $\mu$ for which $\mathcal{C}_{\mu}$ is bounded (compact) from one Dirichlet-type space $\mathcal{D}_{\alpha}$ into another one $\mathcal{D}_{\beta}$. \end{abstract}

\maketitle
\section{Introduction}

Let $\mathbb{D}$ be the unit disk in the complex plane $\mathbb{C}$.  We use $C, C_1, C_2, \cdots$ to denote universal positive constants that might change from one line to another.  For two positive numbers $A, B$, we write
$A \preceq B$, or $A \succeq B$, if there exists a positive constant $C$ independent of $A$ and $B$ such that $A \leq C B$, or $A \geq C B$, respectively. We will write $A \asymp B$ if both $A \preceq B$ and $A \succeq B$.

We denote by $\mathcal{H}(\mathbb{D})$ the class of all analytic functions on $\mathbb{D}$.  Let $0<p<\infty$, the Hardy space $H^p$ is the class of all $f\in \mathcal{H}(\mathbb{D})$  such that
$$\|f\|_{H^p}=\sup_{r\in (0, 1)}M_{p}(r, f)<\infty,$$
where $$M_{p}(r ,f)=\{\frac{1}{2\pi}\int_{0}^{2\pi}|f(re^{i\theta})|^p\,d\theta\}^{\frac{1}{p}},\:0<r<1.$$

For $\alpha\in \mathbb{R}$, the Dirichlet-type space, denoted by $\mathcal{D}_{\alpha}$, is defined as
\begin{equation*}
\mathcal{D}_{\alpha}=\{f=\sum_{n=0}^{\infty}a_nz^n\in \mathcal{H}(\mathbb{D}): \|f\|_{\mathcal{D}_{\alpha}}:=(\sum_{n=0}^{\infty}(n+1)^{1-\alpha} |a_{n}|^2)^{\frac{1}{2}}<\infty\}.\end{equation*}

When $\alpha=0$,  $\mathcal{D}_0$ coincides the classic Dirichlet space $\mathcal{D}$, and when $\alpha=1$, $\mathcal{D}_1$ becomes the Hardy space $H^2$.

The Ces\`aro operator, which is an operator on spaces of analytic functions by its action on the Taylor coefficients, is defined as, for $f=\sum_{n=0}^{\infty}a_n z^n \in \mathcal{H}(\mathbb{D})$,
\begin{equation*} \mathcal{C}(f)(z)=\sum_{n=0}^\infty \left(\frac{1}{n+1}\sum_{k=0}^{n} a_k \right)z^n, z\in \mathbb{D}.
\end{equation*}

The boundedness and compactness of the Ces\`aro operator and its generalizations defined on various spaces of analytic
functions, like Hardy spaces, Bergman spaces and Dirichlet spaces, have attracted many attentions (see for example, \cite{AC}, \cite{AS}, \cite{G}, \cite{MM}, \cite{St}, \cite{S1}, \cite{S2}, \cite{Xiao}, \cite{A} and the references therein).

In this note, we consider the boundedness and compactness of Ces\`aro operator between the Dirichlet-type spaces. We denote $\mathbb{N}_0=\mathbb{N}\cup \{0\}$.  When $0<\alpha<2$.  For $f=\sum_{n=0}^{\infty}a_n z^n \in \mathcal{D}_{\alpha}$, by Cauchy's inequality, we obtain that, for $n\in \mathbb{N}_0$,
\begin{eqnarray}\label{e-1}
\lefteqn{\left |\frac{1}{n+1}\sum_{k=0}^{n} a_k\right|=\frac{1}{n+1}\left|\sum_{k=0}^{n} \left[(k+1)^{\frac{2-\alpha}{4}}a_k\right]
\left[(k+1)^{-\frac{2-\alpha}{4}}\right]\right|}
 \\
&&\leq \frac{1}{n+1}\left[ \sum_{k=0}^{n} (k+1)^{\frac{2-\alpha}{2}} a_k^{2}\right]^{\frac{1}{2}}
\left[\sum_{k=0}^{n}(k+1)^{-\frac{2-\alpha}{2}}\right]^{\frac{1}{2}}. \nonumber
\end{eqnarray}

For $0<\alpha<2$, it is easy to see that
\begin{equation}\label{e-2}
\sum_{k=0}^{n}(k+1)^{-\frac{2-\alpha}{2}}=\sum_{j=1}^{n+1}j^{-\frac{2-\alpha}{2}}\leq \int_{0}^{n+1} t^{-\frac{2-\alpha}{2}}dt=\frac{2}{\alpha}(n+1)^{\frac{\alpha}{2}}.
\end{equation}
Consequently, we get from (\ref{e-1}) and (\ref{e-2}) that
\begin{eqnarray}\label{e-3}
\lefteqn{\|\mathcal{C}(f)\|_{\mathcal{D}_\alpha}=\left[\sum_{n=0}^{\infty}(n+1)^{1-\alpha}\left|\frac{1}{n+1}\sum_{k=0}^{n} a_k\right|^2\right]^{\frac{1}{2}}} \\
&&\leq \left(\frac{2}{\alpha}\right)^{\frac{1}{2}}\left[\sum_{n=0}^{\infty}\frac{1}{(n+1)^{\frac{2+\alpha}{2}}}\sum_{k=0}^{n}(k+1)^{\frac{2-\alpha}{2}}a_k^{2}\right]^{\frac{1}{2}}
\nonumber \\
&&=\left(\frac{2}{\alpha}\right)^{\frac{1}{2}}\left[\sum_{k=0}^{\infty}(k+1)^{1-\alpha}\sum_{n=k}^{\infty}\frac{(k+1)^{\frac{\alpha}{2}}}{(n+1)^{\frac{2+\alpha}{2}}}a_k^{2}\right]^{\frac{1}{2}}.
\nonumber
\end{eqnarray}

We notice that, for $k\in \mathbb{N}_0$,
\begin{eqnarray}\label{e-4}
\sum_{n=k}^{\infty}\frac{(k+1)^{\frac{\alpha}{2}}}{(n+1)^{\frac{2+\alpha}{2}}}&=& \frac{1}{k+1}+\sum_{n=k+1}^{\infty}\frac{(k+1)^{\frac{\alpha}{2}}}{(n+1)^{\frac{2+\alpha}{2}}}\\
&\leq & \frac{1}{k+1}+\int_{k}^{\infty}\frac{(k+1)^{\frac{\alpha}{2}}}{(t+1)^{\frac{2+\alpha}{2}}}\,dt\nonumber \\
&=& \frac{1}{k+1}+\frac{2}{\alpha}\nonumber \\ &\leq& \frac{2+\alpha}{\alpha}.
\nonumber
\end{eqnarray}

It follows from (\ref{e-3}) and (\ref{e-4}) that
\begin{eqnarray}\label{e-5}
\|\mathcal{C}(f)\|_{\mathcal{D}_\alpha}\leq \frac{\sqrt{2(2+\alpha)}}{\alpha}\|f\|_{\mathcal{D}_{\alpha}}.
\nonumber
\end{eqnarray}
This means that $\mathcal{C}: \mathcal{D}_{\alpha} \rightarrow \mathcal{D}_{\alpha}$ is bounded for $0<\alpha<2$. We have proved the following

\begin{proposition}\label{pro-1}
Let $0<\alpha<2$, then the Ces\`aro operator $\mathcal{C}$ is bounded from $\mathcal{D}_{\alpha}$ into itself.
\end{proposition}

It is natural to ask whether the Ces\`aro operator is still bounded from $\mathcal{D}_{\alpha}$ into $\mathcal{D}_{\beta}$, when $\alpha \neq \beta$.

We observe that the Ces\`aro operator $\mathcal{C}$ is not bounded from $\mathcal{D}_{\alpha}$ into $\mathcal{D}_{\beta}$, if $\alpha>\beta$ and $0<\alpha<2$. Actually, if $0<\alpha<2$ and $\alpha>\beta$, let $0<\varepsilon<\alpha$ and set $f= \sum_{n=0}^{\infty}a_n z^n$ with
$$a_n=\sqrt{\frac{\varepsilon}{1+\varepsilon}}(n+1)^{-\frac{2-\alpha+\varepsilon}{2}},\: n\in \mathbb{N}_0.$$
It is easy to see that $$\|f\|_{\mathcal{D}_\alpha}=\sqrt{\frac{\varepsilon}{1+\varepsilon}}[\sum_{n=0}^{\infty}(n+1)^{-1-\varepsilon}]^{\frac{1}{2}}\leq \sqrt{\frac{\varepsilon}{1+\varepsilon}}[1+\int_{1}^{\infty}t^{-1-\varepsilon}\,dt]^{\frac{1}{2}}=1.$$

Since $0<\alpha<2$ and $0<\varepsilon<\alpha$, we see that
\begin{eqnarray}\label{maod}\lefteqn{\|\mathcal{C}(f)\|_{\mathcal{D}_\beta}^2=\frac{\varepsilon}{1+\varepsilon}\sum_{n=0}^{\infty}(n+1)^{1-\beta}\left[\frac{1}{n+1}\sum_{k=0}^{n} (k+1)^{-\frac{2-\alpha+\varepsilon}{2}}\right]^2} \\ &&=\frac{\varepsilon}{1+\varepsilon}
\sum_{n=0}^{\infty}(n+1)^{\alpha-\beta-1-\varepsilon}\left[\sum_{k=0}^{n}\frac{(n+1)^{\frac{-\alpha+\varepsilon}{2}}}{(k+1)^{\frac{2-\alpha+\varepsilon}{2}}}\right]^2 \nonumber \\
&&\geq \frac{\varepsilon}{1+\varepsilon}
\sum_{n=0}^{\infty}(n+1)^{\alpha-\beta-1-\varepsilon}\left[ (n+1)^{\frac{-\alpha+\varepsilon}{2}}\int_{1}^{n+2} \frac{dt}{t^{\frac{2-\alpha+\varepsilon}{2}}} \right]^2 \nonumber \\
&&=\frac{\varepsilon}{1+\varepsilon}
\left(\frac{2}{\alpha-\varepsilon}\right)^2\sum_{n=0}^{\infty}(n+1)^{\alpha-\beta-1-\varepsilon}\left[\left(\frac{n+2}{n+1}\right)^{\frac{\alpha-\varepsilon}{2}}-
\frac{1}{(n+1)^{\frac{\alpha-\varepsilon}{2}}}\right]^2
.\nonumber
\end{eqnarray}

We note that $$\left(\frac{n+2}{n+1}\right)^{\frac{\alpha-\varepsilon}{2}}\rightarrow 1, \: n\rightarrow \infty,$$
and
$$\frac{1}{(n+1)^{\frac{\alpha-\varepsilon}{2}}}\rightarrow 0, \: n\rightarrow \infty.$$

Then we conclude from (\ref{maod}) that there is a constant $\mathcal{N}\in \mathbb{N}$ such that
\begin{eqnarray}\|\mathcal{C}(f)\|_{\mathcal{D}_\beta}^2\geq \frac{\varepsilon}{2(1+\varepsilon)}\left(\frac{2}{\alpha-\varepsilon}\right)^2\sum_{n=\mathcal{N}}^{\infty}(n+1)^{\alpha-\beta-1-\varepsilon}\nonumber.
\end{eqnarray}

If $\mathcal{C}: \mathcal{D}_{\alpha}\rightarrow \mathcal{D}_{\beta}$ is bounded, then there exists a constant $C_1>0$ such that
\begin{equation}\label{c-1}C_1\geq\frac{\|\mathcal{C}(f)\|_{\mathcal{D}_{\beta}}^2}{\|f\|_{\mathcal{D}_\alpha}^2}\geq \frac{\varepsilon}{2(1+\varepsilon)} \left(\frac{2}{\alpha-\varepsilon}\right)^2\sum_{n=\mathcal{N}}^{\infty}(n+1)^{\alpha-\beta-1-\varepsilon}.
\end{equation}
However, when $\varepsilon<\min\{\alpha-\beta, \alpha\}$, we see that $$\sum_{n=\mathcal{N}}^{\infty}(n+1)^{\alpha-\beta-1-\varepsilon}=+\infty.$$

Hence we get that (\ref{c-1}) is a contradiction. This means that the Ces\`aro operator $\mathcal{C}$ is not bounded from $\mathcal{D}_{\alpha}$ into $\mathcal{D}_{\beta}$, if $\alpha>\beta$ and $0<\alpha<2$.

We note that
$$\frac{1}{n+1}=\int_{0}^1t^{n}dt,\: n\in \mathbb{N}_0.$$

Let $\mu$ be a positive Bore measure on $[0, 1)$. For $f=\sum_{n=0}^{\infty}a_n z^n \in \mathcal{H}(\mathbb{D})$, we define the generalized Ces\`aro operators $\mathcal{C}_{\mu}$ as
\begin{equation*} \mathcal{C}_{\mu}(f)(z):=\sum_{n=0}^\infty \left(\mu[n]\sum_{k=0}^n a_k\right)z^n,  z\in \mathbb{D},
\end{equation*}
where $$\mu[n]=\int_{0}^1 t^{n}d\mu(t),\: n\in \mathbb{N}_{0}.$$

In this paper, we first study the question of characterizing measures $\mu$ such that $\mathcal{C}_{\mu}: \mathcal{D}_{\alpha}\rightarrow \mathcal{D}_{\beta}$ is bouded. We obtain a sufficient and necessary condition of $\mu$ for which $\mathcal{C}_{\mu}: \mathcal{D}_{\alpha}\rightarrow \mathcal{D}_{\beta}$ is bouded.

To state our first result, we introduce the notation of generalized Carleson measure on $[0, 1)$.  Let $s>0$ and $\mu$ be a positive Borel measure on $[0,1)$. We say $\mu$ is an $s$-Carleson measure if there is a constant $C_2>0$ such that $$\mu([t, 1))\leq C_2 (1-t)^s,$$ for all $t\in [0, 1)$.

Now we can state the first main result of this paper.
\begin{theorem}\label{main}
Let $0<\alpha, \beta<2$. Then the following statements are equivalent:

{\bf (1)} $\mathcal{C}_{\mu}: \mathcal{D}_{\alpha}\rightarrow \mathcal{D}_{\beta}$ is bounded.

{\bf (2)} $\mu$ is a $[1+\frac{1}{2}(\alpha-\beta)]$-Carleson measure on $[0, 1)$.

{\bf (3)} There is a constant $C_3>0$ such that $$\mu[n] \leq \frac{C_3}{(n+1)^{1+\frac{1}{2}(\alpha-\beta)}},$$ for all $n \in \mathbb{N}_{0}$.
\end{theorem}

The proof of Theorem \ref{main} will be given in the next section. We shall characterize a measure $\mu$ such that $\mathcal{C}_{\mu}: \mathcal{D}_{\alpha}\rightarrow \mathcal{D}_{\beta}$ is compact in the last section.

\section{Proof of Theorem \ref{main}}

In our proof of Theorem \ref{main}, we need the Beta function defined as follows.
$$B(u,v)=\int_{0}^{1}t^{u-1}(1-t)^{v-1}\,dt,\: u>0,v>0.$$

It is known that
$$B(u,v)=\frac{\Gamma(u)\Gamma{(v)}}{\Gamma(u+v)},$$
where $\Gamma$ is the Gamma function, defined as
$$\Gamma(x)=\int_{0}^{\infty}e^{-t} t^{x-1}\,dt,\: x>0.$$
For more detailed introduction to the Beta function and Gamma function, see \cite{W}.

%We start to prove Theorem \ref{main}.  We first show

{\bf (2)$\Rightarrow$(3)}.  We note that {\bf (3)} is obvious when $n=0$.  We get from integration by parts that, for $n(\geq 1)\in \mathbb{N}$,
\begin{eqnarray}
\mu[n]=\int_{0}^1 t^{n}d\mu(t)&=&\mu([0,1))-n\int_{0}^1 t^{n-1}\mu([0, t))dt \nonumber \\
&=& n\int_{0}^1 t^{n-1}\mu([t, 1))dt.\nonumber
\end{eqnarray}

Since $\mu$ is a $[1+\frac{1}{2}(\alpha-\beta)]$-Carleson measure on $[0, 1)$, then we see that there is a constant $C_4>0$ such that
$$\mu([t,1))\leq C_4 (1-t)^{1+\frac{1}{2}(\alpha-\beta)},$$
for all $t\in [0,1).$

It follows that
\begin{eqnarray}\mu[n] &\leq & C_4 n\int_{0}^1 t^{n-1}(1-t)^{1+\frac{1}{2}(\alpha-\beta)}dt\nonumber \\&=&C_4 n\cdot \frac{\Gamma(n)\Gamma(2+\frac{1}{2}(\alpha-\beta))}{\Gamma(n+2+\frac{1}{2}(\alpha-\beta))}\nonumber \\ &\asymp & \frac{1}{(n+1)^{1+\frac{1}{2}(\alpha-\beta)}}.\nonumber \end{eqnarray}

Here we have used the fact that
$$\Gamma(x) = \sqrt{2\pi} x^{x-\frac{1}{2}}e^{-x}[1+r(x)],\,x>0, $$
where $|r(x)|\leq e^{\frac{1}{12x}}-1.$ Hence {\bf (2)$\Rightarrow$(3)} is true.

{\bf (3)$\Rightarrow$(1)}. Let $f=\sum_{n=0}^{\infty}a_n z^n \in \mathcal{D}_{\alpha}$. By Cauchy's inequality, we see from
$$\mu[n]\leq \frac{C_5}{(n+1)^{1+\frac{1}{2}(\alpha-\beta)}}$$ that, for $n\in \mathbb{N}_0$,
\begin{eqnarray}
\lefteqn{\left |\mu[n]\sum_{k=0}^{n}a_k\right|\leq \frac{C_5}{(n+1)^{1+\frac{1}{2}(\alpha-\beta)}}\left |\sum_{k=0}^{n}{a_k}\right|}  \nonumber \\
&=&\frac{C_5}{(n+1)^{1+\frac{1}{2}(\alpha-\beta)}} \left|\sum_{k=0}^{n} \left[(k+1)^{\frac{2-\alpha}{4}} a_k\right]
\left[(k+1)^{-\frac{2-\alpha}{4}}\right]\right|\nonumber \\
&\leq &\frac{C_5}{(n+1)^{1+\frac{1}{2}(\alpha-\beta)}} \left[\sum_{k=0}^{n} (k+1)^{\frac{2-\alpha}{2}} a_k^2\right]^{\frac{1}{2}}
\left[ \sum_{k=0}^{n} (k+1)^{-\frac{2-\alpha}{2}}\right]^{\frac{1}{2}}. \nonumber
\end{eqnarray}
Consequently, we obtain from (\ref{e-2}) that
\begin{eqnarray}
\lefteqn{\|\mathcal{C}_{\mu}(f)\|_{\mathcal{D}_\beta}\leq \left[\sum_{n=0}^{\infty}(n+1)^{1-\beta}\left |\sum_{k=0}^{n} \mu[n]a_k\right|^2\right]^{\frac{1}{2}}}\nonumber \\
&&\leq C_5\left(\frac{2}{\alpha}\right)^{\frac{1}{2}}\left[\sum_{n=0}^{\infty}\frac{1}{(n+1)^{\frac{2+\alpha}{2}}}\sum_{k=0}^{n} (k+1)^{\frac{2-\alpha}{2}} a_k^2\right]^{\frac{1}{2}}
\nonumber \\
&&=C_5\left(\frac{2}{\alpha}\right)^{\frac{1}{2}}\left[\sum_{k=0}^{\infty}(k+1)^{1-\alpha}\sum_{n=k}^{\infty}\frac{(k+1)^{\frac{\alpha}{2}}}{(n+1)^{\frac{2+\alpha}{2}}}a_k^{2}\right]^{\frac{1}{2}}.
\nonumber
\end{eqnarray}

Then it follows from (\ref{e-4}) that
\begin{equation*}
\|\mathcal{C}_{\mu}(f)\|_{\mathcal{D}_\beta}\leq C_5\frac{\sqrt{2(2+\alpha)}}{\alpha}\|f\|_{\mathcal{D}_{\alpha}}.
\end{equation*}
This proves {\bf (3)$\Rightarrow$(1)}.

{\bf (1)$\Rightarrow$(2)}.  We need the following estimate presented in \cite{Zh}.  Let $0<t<1$. For any $c>0$, we have
\begin{equation}\label{est}\sum_{n=1}^{\infty}n^{c-1}t^{2n}\asymp \frac{1}{(1-t^2)^c}.
\end{equation}

For $0<b<1$, let $\mathbf{N}$ be a natural number. We set $\widetilde{f}=\sum_{n=0}^{\infty}\widetilde{a}_n z^n$ with
$$\widetilde{a}_n=\begin{cases}
[\Omega_\mathbf{N}]^{-\frac{1}{2}}b^{n+1}, \; \text{if} \; n\in [0, \mathbf{N}], \\
0, \; \text{if} \; n\geq \mathbf{N}+1,\\
\end{cases}
$$
where
$$\Omega_\mathbf{N}=\sum_{k=0}^{\mathbf{N}}(k+1)^{1-\alpha}b^{2(k+1)}.$$

Then it is easy to see $\|\widetilde{f}\|_{\mathcal{D}_\alpha}=1.$   We set $\mathbf{S}_{\mathbf{N}}=\{k\in \mathbb{N}_0: k\leq \mathbf{N}\}.$ In view of the boundedness of $\mathcal{C}_{\mu}: \mathcal{D}_{\alpha}\rightarrow \mathcal{D}_{\beta}$, we obtain that
\begin{eqnarray}\label{last}
1 &\succeq & \|\mathcal{C}_{\mu}(\widetilde{f})\|_{\mathcal{D}_\beta}^2 =\sum_{n=0}^{\infty}(n+1)^{1-\beta}\left |\sum_{k=0}^{n}\widetilde{a}_k\int_{0}^1t^{n}d\mu(t) \right|^2 \\
&=&[\Omega_\mathbf{N}]^{-1}\sum_{n=0}^{\infty}(n+1)^{1-\beta}\left[\sum_{k=0}^{n}\chi_{\mathbf{S}_{\mathbf{N}}}(k)b^{k+1}\int_{0}^{1}t^{n}d\mu(t)\right]^{{2}}
\nonumber \\
&\geq &[\Omega_\mathbf{N}]^{-1}\sum_{n=0}^{\infty}(n+1)^{1-\beta}\left[\sum_{k=0}^{n}\chi_{\mathbf{S}_{\mathbf{N}}}(k)b^{k+1}\int_{b}^{1}t^{n}d\mu(t))\right]^{{2}} \nonumber \\
&\geq &[\Omega_\mathbf{N}]^{-1}[\mu([b, 1))]^{2}\sum_{n=0}^{\infty}(n+1)^{1-\beta}\left[\sum_{k=0}^{n}\chi_{\mathbf{S}_{\mathbf{N}}}(k)b^{k+1}\cdot b^{n}\right]^{{2}}\nonumber \\
&=&[\Omega_\mathbf{N}]^{-1}[\mu([b, 1))]^{2}\sum_{n=0}^{\infty}(n+1)^{1-\beta}b^{2n}\cdot \left[\sum_{k=0}^{n}\chi_{\mathbf{S}_{\mathbf{N}}}(k)b^{k+1}\right]^{{2}}.\nonumber
\end{eqnarray}

On the other hand, we note that, when $n\leq  \mathbf{N}$,
\begin{equation*}\label{n-1}\sum_{k=0}^{n}\chi_{\mathbf{S}_{\mathbf{N}}}(k)b^{k+1}\geq (n+1)b^{n+1}.\end{equation*}

Then we get that
\begin{eqnarray}
\sum_{n=0}^{\infty}(n+1)^{1-\beta}b^{2n}\cdot \left[\sum_{k=0}^{n}\chi_{\mathbf{S}_{\mathbf{N}}}(k)b^{k+1}\right]^{{2}}\geq \sum_{n=0}^{\mathbf{N}}(n+1)^{3-\beta}b^{4n+2}.
\nonumber
\end{eqnarray}
It follows from (\ref{last}) that
\begin{equation}\label{end}
1\succeq [\Omega_\mathbf{N}]^{-1}[\mu([b, 1))]^{2}\sum_{n=0}^{\mathbf{N}}(n+1)^{3-\beta}b^{4n+2}.
\end{equation}

Taking $\mathbf{N}\rightarrow \infty$ in (\ref{end}), we see that
$$[\mu([b, 1))]^{2}\sum_{n=0}^{\infty}(n+1)^{3-\beta}b^{4n+2} \preceq \sum_{n=0}^{\infty}(n+1)^{1-\alpha}b^{2(n+1)},$$
for all $b\in [0,1)$. Then we conclude from (\ref{est}) that
$$[\mu([b, 1))]^{2}\frac{1}{(1-b^2)^{4-\beta}}\preceq \frac{1}{(1-b^2)^{2-\alpha}}.$$

This implies that
$$\mu([b, 1))\preceq (1-b^2)^{1+\frac{1}{2}(\alpha-\beta)},$$
for all $0<b<1$. It follows that $\mu$ is a $[1+\frac{1}{2}(\alpha-\beta)]$-Carleson measure on $[0, 1)$ and {\bf (1)$\Rightarrow$(2)} is proved.
The proof of Theorem \ref{main} is now finished.

\section{Compactness of the generalized Ces\`aro operators on Dirichlet-type spaces}
%In this section, we characterize measures $\mu$ such that $\mathcal{C}_{\mu}: \mathcal{D}_{\alpha}\rightarrow \mathcal{D}_{\beta}$ is compact.
For $0<s<\infty$, we say a positive Borel measure $\mu$ on $[0, 1)$ is a vanishing $s$-Carleson measure, if $\mu$ is an $s$-Carleson measure and satisfies that
$$\lim_{t\rightarrow 1^{-}}\frac{\mu([t, 1))}{(1-t)^s}=0.$$

The following theorem is the main result of this section.
\begin{theorem}\label{main-1}
Let $0<\alpha, \beta<2$. Then the following statements are equivalent:

{\bf (1)} $\mathcal{C}_{\mu}: \mathcal{D}_{\alpha}\rightarrow \mathcal{D}_{\beta}$ is compact.

{\bf (2)} $\mu$ is a vanishing $[1+\frac{1}{2}(\alpha-\beta)]$-Carleson measure on $[0, 1)$.

{\bf (3)} $$\mu[n]=o\left((n+1)^{-1-\frac{1}{2}(\alpha-\beta)}\right), \: n\rightarrow \infty.$$
\end{theorem}

\begin{proof}[Proof of Theorem \ref{main-1}]
First note that, by minor modifications of the arguments of
{\bf (2)$\Rightarrow$(3)} in the proof of Theorem \ref{main}, we can similarly show {\bf (2)$\Rightarrow$(3)} of Theorem \ref{main-1}.

We proceed to prove {\bf (3)$\Rightarrow$(1)}. For any $f=\sum_{n=0}^{\infty}a_n z^n \in \mathcal{D}_{\alpha}$. Let $\mathfrak{N}\in \mathbb{N}$, we consider
\begin{equation*} \mathcal{C}_{\mu}^{[\mathfrak{N}]}(f)(z):=\sum_{n=0}^{\mathfrak{N}} \left(\mu[n]\sum_{k=0}^n a_k\right)z^n,  z\in \mathbb{D}.
\end{equation*}
Then we see that $\mathcal{C}_{\mu}^{[\mathfrak{N}]}$ is a finite rank operator, hence $\mathcal{C}_{\mu}^{[\mathfrak{N}]}$ is compact from $\mathcal{D}_{\alpha}$ into
$\mathcal{D}_{\beta}$.

In view of $$\mu[n]=o\left((n+1)^{-1-\frac{1}{2}(\alpha-\beta)}\right), \: n\rightarrow \infty.$$
We see that, for any $\epsilon>0$, there is an $N_0 \in \mathbb{N}$ such that
$$\mu[n]\leq \epsilon (n+1)^{-1-\frac{1}{2}(\alpha-\beta)},$$
for all $n>N_0$.

Note that
\begin{eqnarray}\|(\mathcal{C}_{\mu}-\mathcal{C}_{\mu}^{[\mathfrak{N}]})(f)\|_{\mathcal{D}_{\beta}}^2=\sum_{n=\mathfrak{N}+1}^{\infty}(n+1)^{1-\beta} \left|\mu[n]\sum_{k=0}^n a_k\right|^2.\nonumber
\end{eqnarray}

When $\mathfrak{N}>N_0$, we get that
\begin{eqnarray}\|(\mathcal{C}_{\mu}-\mathcal{C}_{\mu}^{[\mathfrak{N}]})(f)\|_{\mathcal{D}_{\beta}}^2\preceq {\epsilon}^2 \sum_{n=\mathfrak{N}+1}^{\infty}(n+1)^{1-\beta} \left|\frac{1}{(n+1)^{1+\frac{1}{2}(\alpha-\beta)}}\sum_{k=0}^n a_k\right|^2.\nonumber\end{eqnarray}

Consequently, by using the arguments of {\bf (3)$\Rightarrow$(1)} in the proof of Theorem \ref{main}, we see that
\begin{eqnarray}\|(\mathcal{C}_{\mu}-\mathcal{C}_{\mu}^{[\mathfrak{N}]})(f)\|_{\mathcal{D}_{\beta}}^2\preceq {\epsilon}^2 \|f\|_{\mathcal{D}_{\alpha}}^2,\nonumber\end{eqnarray}
holds for any $f\in\mathcal{D}_{\alpha}$.
Hence we see
\begin{eqnarray}\|\mathcal{C}_{\mu}-\mathcal{C}_{\mu}^{[\mathfrak{N}]}\|_{\mathcal{D}_{\alpha}\rightarrow \mathcal{D}_{\beta}}\preceq {\epsilon},\nonumber\end{eqnarray}
when $\mathfrak{N}>N_0$.
Here \begin{eqnarray}\|T\|_{\mathcal{D}_{\alpha}\rightarrow \mathcal{D}_{\beta}}=\sup_{f(\neq 0)\in \mathcal{D}_{\alpha}}\frac{\|T(f)\|_{\mathcal{D}_{\beta}}}{\|f\|_{\mathcal{D}_{\alpha}}},\end{eqnarray}
where $T$ is a linear bounded operator from $\mathcal{D}_{\alpha}$ into
$\mathcal{D}_{\beta}$.  This means that $\mathcal{C}_{\mu}$ is compact from $\mathcal{D}_{\alpha}$ into
$\mathcal{D}_{\beta}$ and {\bf (3)$\Rightarrow$(1)} is proved.

Finally, we show that {\bf (1)$\Rightarrow$(2)}. For $0<b<1$.  We set  $\widehat{f}_b=\sum_{n=0}^{\infty}\widehat{a}_n z^n$ with
$$\widehat{a}_n=(1-b^2)^{\frac{2-\alpha}{2}}b^{n+1},\: n\in \mathbb{N}_0.
$$

We see from (\ref{est}) that $\|\widehat{f}_b\|_{\mathcal{D}_{\alpha}}\asymp1.$ By the fact that $\lim\limits_{b\rightarrow 1^{-1}}\widehat{f}_b=0$ for any $z\in \mathbb{D}$, we conclude that
$\widehat{f}_b$ is convergent weakly to $0$ in $\mathcal{D}_{\alpha}$ as $b\rightarrow 1^{-}$. Since $\mathcal{C}_{\mu}: \mathcal{D}_{\alpha} \rightarrow \mathcal{D}_{\beta}$ is compact, we see
\begin{equation}\label{com-0}\lim_{b\rightarrow {1^{-}}} \|\mathcal{C}_{\mu}(\widehat{f}_b)\|_{\mathcal{D}_{\beta}}=0.\end{equation}

On the other hand, we have
\begin{eqnarray}\label{com-1}
 \|\mathcal{C}_{\mu}(\widehat{f}_b)\|_{\mathcal{D}_\beta}^2 &=&\sum_{n=0}^{\infty}(n+1)^{1-\beta}\left |\sum_{k=0}^{n}\widehat{a}_k\int_{0}^1t^{n}d\mu(t) \right|^2 \\
&=&(1-b^2)^{2-\alpha}\sum_{n=0}^{\infty}(n+1)^{1-\beta}\left[\sum_{k=0}^{n}b^{k+1}\int_{0}^{1}t^{n}d\mu(t)\right]^{{2}}
\nonumber \\
&\geq &(1-b^2)^{2-\alpha}\sum_{n=0}^{\infty}(n+1)^{1-\beta}\left[\sum_{k=0}^{n}b^{k+1}\int_{b}^{1}t^{n}d\mu(t))\right]^{{2}} \nonumber \\
&\geq &(1-b^2)^{2-\alpha}[\mu([b, 1))]^{2}\sum_{n=0}^{\infty}(n+1)^{1-\beta}\left[\sum_{k=0}^{n}b^{k+1}\cdot b^{n}\right]^{{2}}\nonumber \\
&=&(1-b^2)^{2-\alpha}[\mu([b, 1))]^{2}\sum_{n=0}^{\infty}(n+1)^{1-\beta}b^{2n}\cdot \left(\sum_{k=0}^{n}b^{k+1}\right)^{{2}}.\nonumber
\end{eqnarray}

Also, we have
\begin{eqnarray}\label{com-2}
\lefteqn{\sum_{n=0}^{\infty}(n+1)^{1-\beta}b^{2n}\cdot \left(\sum_{k=0}^{n}b^{k+1}\right)^{{2}}}\\&&\geq \sum_{n=0}^{\infty}(n+1)^{1-\beta}b^{2n}\cdot[(n+1)b^{n+1}]^{{2}}\nonumber \\
&&=\sum_{n=0}^{\infty}(n+1)^{3-\beta}b^{4n+2}\asymp \frac{1}{(1-b^2)^{4-\beta}}\nonumber.
\end{eqnarray}
Combining (\ref{com-1}) and (\ref{com-2}), we see
\begin{eqnarray}
 \mu([b, 1))\preceq \|\mathcal{C}_{\mu}(\widehat{f}_b)\|_{\mathcal{D}_\beta}(1-b^2)^{1+\frac{1}{2}(\alpha-\beta)}\nonumber.
\end{eqnarray}

It follows from (\ref{com-0}) that
 \begin{eqnarray}
\lim_{b\rightarrow 1^{-}}\frac{ \mu([b, 1))}{(1-b)^{1+\frac{1}{2}(\alpha-\beta)}}=0\nonumber.
\end{eqnarray}
This proves {\bf (1)$\Rightarrow$(2)} and the proof of Theorem \ref{main-1} is complete.
\end{proof}

%We end this paper by the following general
%\begin{problem}
%Characterize measures $\mu$ such that $\mathcal{C}_{\mu}$ boundedly (compactly) maps one analytic function space $A$ into another space $B$.
%\end{problem} For example, $A, B$ can be considered as the Hardy spaces, Bergman spaces, or Bloch spaces and so on.

\section*{acknowledgement}
We thank the referee for useful suggestions on this paper.


\end{document}